\documentclass{amsart}

\usepackage[T1]{fontenc}
\usepackage{amssymb,commath,graphicx,amsfonts,mathtools,mathrsfs,xcolor}
\usepackage[hidelinks]{hyperref}
\hypersetup{
    colorlinks = false,
    citebordercolor = white,
    linkbordercolor = white,
}
\usepackage[percent]{overpic}

\usepackage{lmodern}

\newtheorem{ut}{Theorem}
\numberwithin{ut}{section}
\numberwithin{equation}{section}

\newtheorem{ul}[ut]{Lemma}
\newtheorem{uc}[ut]{Corollary}
\newtheorem{ucl}[ut]{Claim}

\theoremstyle{definition}
\newtheorem{ur}[ut]{Remark}

\begin{document}

\title{Distinguishing endpoint sets from Erd\H{o}s space}

\address{Department of Mathematics, Auburn University at Montgomery, Montgomery AL 36117, United States of America}
\email{dsl0003@auburn.edu}
\author{DAVID S. LIPHAM}

\begin{abstract}
 We prove that the set of all endpoints of the Julia set of $f(z)=\exp(z)-1$ which escape to infinity under iteration of $f$ is not homeomorphic to the rational Hilbert space $\mathfrak E$. As a corollary, we show that the  set of all points $z\in \mathbb C$ whose orbits either escape to $\infty$ or attract to $0$ is path-connected.  We extend these results to many other functions in the exponential family.
 \end{abstract}

\maketitle

\section{Introduction}

The exponential family $f_a(z) = e^z + a$; $a\in \mathbb C$, is the most studied
family of functions in the theory of the iteration of transcendental entire functions. For any parameter $a\in \mathbb C$, the Julia set $J(f_a)$ is known to be equal to the closure of the escaping set $I(f_a):=\{z\in \mathbb C:f_a^n(z)\to\infty\}$ \cite{el}.  
And when $a$ belongs to the Fatou set $F(f_a)$ (e.g.\ when $a\in (-\infty,-1]$), the Julia set $J(f_a)$  can be written as a union of uncountably many disjoint curves and endpoints \cite{dev,vas}.  
  A point $z\in J(f_a)$ is on a \textit{curve}  if there exists an arc $\alpha:[-1,1]\hookrightarrow I(f_a)$ such that $\alpha(0)= z$. A point $z_0\in J(f_a)$ is an \textit{endpoint} if $z_0$ is not on a curve and there is an arc $\alpha:[0,1]\hookrightarrow J(f_a)$ with $\alpha(0)=z_0$ and $\alpha(t)\in I(f_a)$ for all $t\in (0,1]$. The set of all endpoints of $J(f_a)$ is denoted $E(f_a)$.




The first  to study the surprising topological properties of the endpoints was Mayer in 1988. He proved that  $\infty$ is an explosion point for $E(f_a)$ for all  attracting fixed point parameters $a\in (-\infty,-1)$ \cite{may}. That is,  $E(f_a)$ is totally separated but its union with $\infty$ is a connected set.  A similarly paradoxical result is due to  McMullen 1987 and Karpinska 1999: If $a\in (-\infty,-1)$ then the Hausdorff dimension of $E(f_{a})$ is two \cite{mcm}, but the set of curves has Hausdorff dimension one \cite{kar}.


Alhabib and Rempe extended Mayer's result in 2016 by focusing on the endpoints of $I(f_a)$. They proved that  $\infty$ is an explosion point for  the \textit{escaping endpoint set} 
$$\dot E(f_a):= E(f_a)\cap I(f_a),$$ as well as for $E(f_a)$,  for every Fatou parameter $a\in F(f_a)$.  The set of non-escaping endpoints $E(f_a)\setminus I(f_a)$ is very different in this regard. Its union with $\infty$ is totally separated \cite{vas} and zero-dimensional \cite{lip2}. 

Conjugacy between escaping sets \cite{rem2} implies that $\dot E(f_a)$ and $\dot E(f_b)$ are topologically equivalent when $a$ and $b$ are Fatou parameters. The primary goal of this paper is to  distinguish   the space $\dot E(f_a)$; $a\in F(f_a)$, from a certain   line-free subgroup of the Hilbert space $\ell^2$.  For comparison, the entire endpoint set  $E(f_a)$ is homeomorphic to  \textit{complete Erd\H{o}s space} $$\mathfrak E_{\mathrm{c}}:=\{\mathbf{x}\in \ell^2:x_n\in \mathbb R\setminus \mathbb Q\text{ for each }n<\omega\}$$ for every  $a\in (-\infty,-1]$ \cite{31}. The space $\dot E(f_a)$ is not homeomorphic to $ \mathfrak E_{\mathrm{c}}$ \cite{lip}, but  has many of the same topological properties as   (the plain) \textit{Erd\H{o}s space} $$\mathfrak E:=\{\mathbf{x}\in \ell^2:x_n\in \mathbb Q\text{ for each }n<\omega\}.$$ For example, $\dot E(f_a)$ and $\mathfrak E$ are  both  one-dimensional, almost zero-dimensional $F_{\sigma\delta}$-spaces which are nowhere  $G_{\delta\sigma}$; see  \cite{lip,lip4}. Moreover $\dot E(f_a)$ contains a dense copy of $\mathfrak E$ in the form of all endpoints that escape to infinity in the imaginary direction \cite{lip3}.  However, in this paper we will show that the two spaces are not equivalent. This provides a negative answer  to  \cite[Question 1]{lip} and  suggests that $\dot E(f_a)$ is a fundamental object    between the `rational and irrational Hilbert spaces' $\mathfrak E$ and $\mathfrak E_{\mathrm{c}}$. It is still unknown whether $\dot E(f_a)$ is a topological group, or is at least homogeneous.

The key property distinguishing $\dot E(f_a)$ from $\mathfrak E$ will involve the notion of a C-set. A \textit{$C$-set} in a topological space $X$ is an intersection of clopen subsets of $X$. Note that for every rational number $q\in \mathbb Q$ the set $\{\mathbf x \in \mathfrak E:x_0=q\}$ is a nowhere dense C-set in $\mathfrak E$ because  the $\ell^2$-norm topology on $\mathfrak E$ is finer than the zero-dimensional topology that $\mathfrak E$ inherits from $\mathbb Q ^\omega$. Thus $\mathfrak E$ can be written as a countable union of nowhere dense C-sets. On the other hand:

\begin{ut}\label{t1}If $a\in F(f_a)$ then $\dot E(f_a)$ cannot be written as a countable union of nowhere dense C-sets. 
\end{ut}

\begin{uc}\label{t2}If $a\in F(f_a)$ then $\dot E(f_a)\not \simeq \mathfrak E$.
\end{uc}

 Our proof of Theorem 1.1 will involve constructing  simple closed (Jordan) curves in $\mathbb C$ whose intersections with  $J(f_{-1})$ are contained in $I(f_{-1})$.


\begin{ut}\label{t3}If $a\in (-\infty,-1]$ then each point of $\mathbb C$ can be separated from $\infty$ by a simple closed (Jordan) curve in $F(f_a)\cup I(f_a)$. 
\end{ut}

Note that the curve in Theorem 1.3 avoids all non-escaping endpoints of $J(f_a)$.  

\begin{uc}\label{t4}If $a\in (-\infty,-1]$ then $F(f_a)\cup I(f_a)$ is path-connected.\end{uc}

We observe that  $F(f_{-1})\cup I(f_{-1})$ is simply the set of all points $z\in \mathbb C$ such that $f_{-1}^n(z)\to 0$ or $f_{-1}^n(z)\to\infty$. 







 \section{Preliminaries}
 \subsection{Outline of paper}

 We will prove Theorem 1.1 and Corollary 1.2  for the particular mapping  $$f(z):=f_{-1}(z)=e^z-1.$$ The proof for all other Fatou parameters will then follow from the conjugacy  \cite[Theorem 1.2]{rem2}. Theorem 1.3 and Corollary 1.4 will essentially follow from the case $a=-1$ as well.

Instead of working directly in the Julia set $J(f)\subset \mathbb C$, we will work in a topologically equivalent subset of $\mathbb R^2$ known as a brush. In this section we will give the definition of a brush and its connection  to exponential Julia sets. We will then define a specific brush which is homeomorphic to $J(f)$, together with a mapping which models $f$.

In Section 3 we will prove three easy lemmas for the model mapping. In Section 4 we will show that certain endpoint sets in brushes cannot be written as countable unions of nowhere dense C-sets, given the existence of certain Jordan curves. In Section 5 we will prove Theorem 1.1 by showing that such curves exist with respect to our model of $\dot E(f)$.  Theorem 1.1 is strengthened by Remark 6.2 in Section 6.  Finally, the proofs of Theorem 1.3 and Corollary 1.4 are given in Section 7. 

\subsection{Brushes and Cantor bouquets}Let $\mathbb P=\mathbb R\setminus \mathbb Q$. A \textit{brush} is a closed subset of $\mathbb R^2$ of the form $$B=\bigcup_{y\in Y} [t_y,\infty)\times \{y\},$$ where $Y\subset \mathbb P$ and  $t_y\in \mathbb R$.   The \textit{endpoints} of $B$ are  the points $\langle t_y,y\rangle$. 

When $a\in (-\infty,-1]$, the Julia set $J(f_a)$ is a \textit{Cantor bouquet} \cite{aa,rem} which is homeomorphic to a brush with a dense set of endpoints. For all other Fatou parameters, $J(f_a)$ is a \textit{pinched Cantor bouquet} that is homeomorphic to a brush modulo a closed equivalence relation on its set of endpoints; see \cite[Corollary 9.3]{rem2} and \cite{geo}. The relation establishes a one-to-one correspondence between endpoints of $J(f)$ and equivalence classes of endpoints of $B$ \cite[Section 3.2]{lip2}.

\subsection{Brush model of $J(f)$}We will now define a mapping $\mathcal F$ and a brush $J(\mathcal F)$ that model $f$ and $J(\mathcal F)$. These objects were studied extensively in \cite{rem2,rem}.

Let  $\mathbb Z ^\omega$ denote the space of  integer sequences $\underline{s}=s_0s_1s_2\ldots$ in the product (or lexicographic order) topology.  Define 
$\mathcal F:[0,\infty)\times \mathbb Z ^\omega\to \mathbb R\times \mathbb Z ^\omega$ by $$\langle t,\underline{s}\rangle\mapsto \langle F(t)-2\pi|s_0|,\sigma(\underline{s})\rangle,$$ where $F(t)=e^t-1$ and $\sigma$ is the shift map on $\mathbb Z ^\omega$; i.e.\  $$\sigma(s_0s_1s_2\ldots)=s_1s_2s_3\ldots.$$    
For each $x=\langle t,\underline{s}\rangle\in [0,\infty)\times \mathbb Z ^\omega$ put $T(x)=t$ and $\underline s(x)=\underline{s}$. The integer sequence $\underline{s}(x)$ is called the \textit{external address} of $x$.

Define  $$J(\mathcal F)=\{x\in [0,\infty)\times \mathbb Z ^\omega:T(\mathcal F^n(x))\geq 0\text{ for all }n\geq 0\}.$$ Continuity of $\mathcal F$ implies that  $J(\mathcal F)$ is closed in  $[0,\infty)\times \mathbb Z ^\omega$. Now let $$\mathbb S=\{\underline s\in \mathbb Z^\omega:\text{ there exists } t\geq 0\text{ such that }\langle t,\underline{s}\rangle\in J(\mathcal F)\},$$  and for each $\underline s\in \mathbb S$ put $t_{\underline{s}}=\min\{t\geq 0:\langle t,\underline{s}\rangle\in J(\mathcal F)\}.$ Then $$J(\mathcal F)=\bigcup_{\underline s\in \mathbb S} [t_{\underline s},\infty)\times \{\underline s\}.$$  Finally,   $\mathbb Z^\omega$ can be identified with  $\mathbb P$ using an order isomorphism between  $\mathbb Z ^\omega$ in the lexicographic ordering and $\mathbb P$ in the real  ordering \cite[Observation 3.2]{rem}. Under this identification,  $J(\mathcal F)\subset \mathbb R\times \mathbb P$. Hence $J(\mathcal F)$ is a brush.

 \begin{figure}[h]
 \includegraphics[scale=0.33]{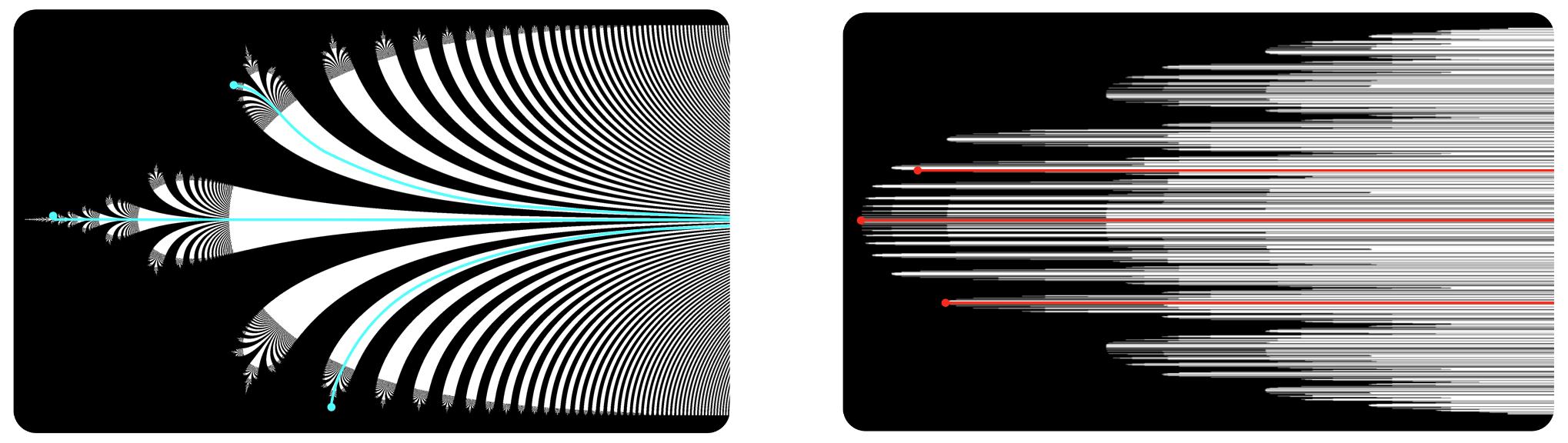}
 \caption{Three endpoints and curves in $J(f)$ (left) and in $J(\mathcal F)$ (right).}
 \end{figure}

The restrictions $\mathcal F\restriction J(\mathcal F)$ and $f\restriction J(f)$ are topologically conjugate \cite[Section 9]{rem2}. Hence $J(\mathcal F)$ is homeomorphic to $J(f)$, and the action of $\mathcal F$ on $J(\mathcal F)$ captures the essential dynamics of $f$ on $J(f)$; see also  \cite[p.74]{rem}.

\section{Lemmas  for $\mathcal F$} 

In order to prove Theorem 1.1 we need a few basic lemmas regarding the model  above. Define $F^{-1}(t)=\ln(t+1)$ for $t\geq 0$, so that  $F^{-1}$ is the inverse of $F$.   For each $n\geq 1$, the $n$-fold composition of $F^{-1}$  is denoted $F^{-n}$.

\begin{ul}[Iterating between double squares]\label{t5}Let $x\in J(\mathcal F)$. If $k\geq 5$ and $$T(\mathcal F^{2k^2}(x))\geq F^{k^2}(1),$$ then $T(\mathcal F^{n}(x))\geq F^{k}(1)$ for all $n\in [2(k-1)^2,2k^2]$.\end{ul}

\begin{proof}Suppose $k\geq 5$ and $T(\mathcal F^{2k^2}(x))\geq F^{k^2}(1)$. Let $n\in [2(k-1)^2,2k^2]$. Then there exists $i\leq 4k-2$ such that $n=2k^2-i$.  We have 
$$F^i(T(\mathcal F^{n}(x)))=F^i(T(\mathcal F^{2k^2-i}(x)))\geq T(\mathcal F ^i (\mathcal F^{2k^2-i}(x)))=T(\mathcal F ^{2k^2}(x))\geq F^{k^2}(1).$$
Applying $F^{-i}$ to each side of the inequality shows that $$T(\mathcal F^{n}(x))\geq F^{k^2-i}(1)> F^{k}(1),$$ where we used  the fact   $k^2-i>k$ for all $k\geq 5$ and $i\leq 4k-2$.
\end{proof}

\begin{ul}[Forward stretching]\label{t6}Let $x,y\in J(\mathcal F)$.  If $\underline s(y)=\underline s(x)$ and $T(y)>T(x)$, then for every $n\geq 1$ we have $T(\mathcal F^n(y))\geq F^n(T(y)-T(x)).$
\end{ul}

\begin{proof}Suppose $\underline s(y)=\underline s(x)$ and $T(y)> T(x)$. Let $\varepsilon=T(y)-T(x)$ and $\delta=T(\mathcal F^n(y))-T(\mathcal F^n(x))$.  Note that  $\delta>0$.  So $$T(y)\leq T(x)+F^{-n}(\delta)$$ by \cite[Observation 3.9]{rem}.  Therefore  $\delta \geq F^n(\varepsilon)$.  We have $$T(\mathcal F^n(y))\geq T(\mathcal F^n(y))-T(\mathcal F^n(x))=\delta \geq F^n(\varepsilon) =F^n(T(y)-T(x))$$ as desired.
\end{proof}

Lemma \ref{t5} implies that  $T(\mathcal F^n(x))\to\infty$  if the double square iterates $T(\mathcal F^{2k^2}(x))$   increase at a sufficient rate. Lemma \ref{t6} will be used to  show that  $T(\mathcal F^{2k^2}(x))$ increases at that rate for certain points $x\in J(\mathcal F)$.

\begin{ul}[Logarithmic orbit]\label{t7}$F^{-n}(1)<3/n$ for all $n\geq 1$.\end{ul}

\begin{proof}The proof is by induction on $n$. If $n=1$ then we have $F^{-n}(1)=F^{-1}(1)=\ln(2)<3=3/n$.  Now suppose the inequality holds for a given $n$.  Then $$ F^{-(n+1)}(1)=F^{-1}(F^{-n}(1))<F^{-1}(3/n)=\ln\left(\frac{3}{n}+1\right)<\frac{3}{n+1}$$ by  calculus.\end{proof}

Although the series $\sum_{k=1}^\infty F^{-k}(1)$   diverges, Lemma \ref{t7} implies that $\sum_{k=1}^\infty F^{-k^2}(1)$ converges (to something less than $\pi ^2/2$). 

\section{Jordan curve lemma}

The following   topological lemma  will also be required to prove Theorem \ref{t1}. Recall that a \textit{C-set} in a  space $X$ is an intersection of clopen subsets of $X$. 

 \begin{ul}\label{lem}Let $$B=\bigcup_{y\in Y} [t_y,\infty)\times \{y\}$$  be a brush, and let $E=\{\langle t_y,y\rangle:y\in Y\}$ denote the set of endpoints of $B$. Suppose $\tilde E\subset E$ and  $\tilde E\cup \{\infty\}$ is connected.  If there exists an open set $U\subset \mathbb R ^2$ such that
\begin{itemize}
\item $U\cap \tilde E\neq\varnothing$, 
\item  $\partial U$ is a simple closed curve, and 
\item $\partial U\cap E\subset \tilde E$, 
\end{itemize} then $\tilde E$ cannot be written as a countable union of nowhere dense C-sets. 
\end{ul}

\begin{proof}Suppose that $\tilde E$, $U$, and $\beta:=\partial U$ satisfy all  of the hypotheses. Let $$A:=\{y\in Y: \beta\text{ contains an interval of }[t_y,\infty)\times \{y\}\}.$$ Since every collection of pairwise disjoint arcs of a simple closed curve is countable, $A$ is countable. So $\tilde E\cap (\mathbb R\times A)$ is countable. Observe also that since  $B$ is closed in $ \mathbb R\times Y$ and $Y\subset \mathbb P$, the space $E$  has a neighborhood basis of C-sets  of the form $E\cap (-\infty,t]\times W$, where $t\in \mathbb R$ and $W$ is clopen in $Y$. 
 By \cite[Theorem 4.7]{coh} and the assumption that $\tilde E\cup \{\infty\}$ is connected,   we see that $\tilde E\setminus (\mathbb R\times A)\cup \{\infty\}$ is connected. 
 
Now aiming for a contradiction, suppose that  $\tilde E=\bigcup \{C_n:n<\omega\}$ where each $C_n$ is a nowhere dense C-set in $\tilde E$. Let  $V$ be the unbounded component of $\mathbb R^2 \setminus \beta$.  
Since $\tilde E\setminus (\mathbb R\times A)\cup \{\infty\}$ is connected, $$\tau:=\overline{U\cap \tilde E}\cap \overline{V\cap \tilde E}\cap \tilde E\setminus (\mathbb R\times A)$$ is  non-empty.  Note also that $\tau$ is a relatively closed subset of  $$\beta\cap \tilde E\setminus (\mathbb R\times A)=\beta\cap E\setminus (\mathbb R\times A),$$ which is a $G_\delta$-subset of $E$.   And $E$ is a $G_\delta$-subset of $\mathbb R \times \mathbb P$ because $B$ is closed and $B\setminus E$ is the union of countably many closed sets $B+\langle \frac{1}{n},0\rangle$.   Therefore $\tau$ is completely metrizable.  
By Baire's theorem there is an open rectangle $(x_1,x_2)\times (y_1,y_2)$ and $n<\omega$ such that $\varnothing\neq \tau\cap ((x_1,x_2)\times (y_1,y_2))\subset C_n.$
Let $$\langle t_y,y\rangle\in \tau\cap ((x_1,x_2)\times (y_1,y_2)).$$ See Figure 2.
  \begin{figure}
 \includegraphics[scale=0.27]{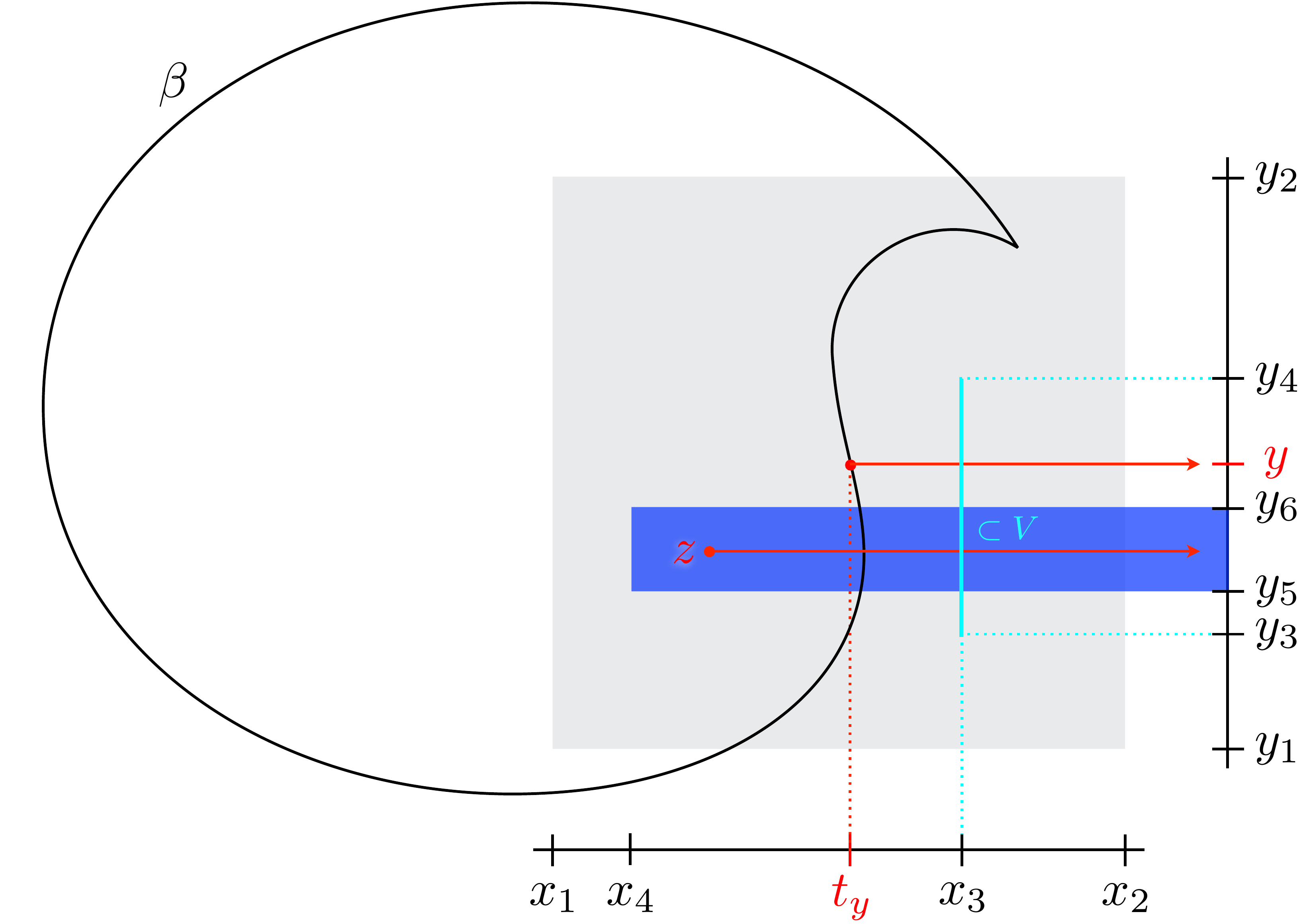}
 \caption{Construction of  clopen set in Lemma \ref{lem}.}
 \end{figure}
Since $y\notin A$, there exist $x_3\in (t_y,x_2)$ and $y_3,y_4\in (y_1,y_2)\cap \mathbb Q$ such that $y_3<y<y_4$  and $\{x_3\}\times [y_3,y_4]\subset \mathbb R ^2\setminus \beta$. Without loss of generality, assume $\{x_3\}\times [y_3,y_4]\subset V$.

Note that $U\cap\tilde E\cap ((x_1,x_3)\times (y_3,y_4))\neq\varnothing$ because $$\langle t_y,y\rangle\in \overline{U\cap \tilde E}\cap ((x_1,x_3)\times (y_3,y_4)). $$ Further,   $C_n\cup (\tilde E\cap (\mathbb R\times A))$ is nowhere dense in $\tilde E$, so there exists $$z\in U\cap\tilde E\cap ((x_1,x_3)\times (y_3,y_4))\setminus (C_n\cup (\mathbb R\times A)).$$  Since $z$ is an endpoint and $B$ is closed in $\mathbb R\times Y$, we can find $y_5,y_6\in (y_3,y_4)\cap \mathbb Q$ and $x_4>x_1$ such that $y_5<y_6$, 
$z\in(x_4,\infty)\times (y_5,y_6)$, and 
$$(\{x_4\}\times [y_5,y_6])\cup ([x_4,\infty)\times \{y_5,y_6\})\cap B=\varnothing.$$   Further,  since $z\notin C_n$ there is a relatively clopen subset $O$ of $\tilde E\setminus (\mathbb R\times A)$ such that $z\in O$ and $C_n\cap O=\varnothing$. Then $$U\cap O \cap ([x_4,x_3]\times [y_5,y_6])$$ is a non-empty bounded  clopen subset of $\tilde E\setminus (\mathbb R\times A)$. This contradicts the previously established fact that  $\tilde E\setminus (\mathbb R\times A)\cup \{\infty\}$ is connected.
 \end{proof}

\section{Proof of Theorem \ref{t1}}

We are now ready to prove Theorem 1.1. 

As indicated in Section 2.2, we will identify $\mathbb Z ^\omega$ with $\mathbb P$, so that $$J(\mathcal F)\subset \mathbb R\times \mathbb P\subset \mathbb R ^2.$$Let $E(\mathcal F)=\{\langle t_{\underline{s}},\underline{s}\rangle:s\in \mathbb S\}$ denote the set of endpoints of $J(\mathcal F)$. Define \begin{align*}
I(\mathcal F)&=\{x\in J(\mathcal F):T(\mathcal F^n(x))\to\infty\} \text{ and}\\
\tilde{E}(\mathcal F)&=  I(\mathcal F)\cap E(\mathcal F).
\end{align*}  
The conjugacy  \cite[Theorem 9.1]{rem2} shows that  $\tilde{E}(\mathcal F) \simeq \dot E(f)$, and $\tilde{E}(\mathcal F)\cup \{\infty\}$ is connected by \cite[Theorem 3.4]{rem}.  Thus to reach the conclusion that $\dot E(f)$ cannot be written as a countable union of nowhere dense C-sets,  by Lemma \ref{lem} we only need to  find a simple closed curve $\beta\subset \mathbb R ^2$ such that:
\begin{itemize}
\item if $U$ is the bounded component of $\mathbb R ^2\setminus \beta$ then $U\cap \tilde E(\mathcal F)\neq\varnothing$, and
\item $\beta\cap J(\mathcal F)\subset I(\mathcal F)$ (in particular, $\beta\cap E(\mathcal F)\subset \tilde E(\mathcal F)$).
\end{itemize}

The $\beta$ that we construct will essentially be the boundary of a union of rectangular regions in $\mathbb R ^2$. We recursively define the collections of rectangles (or ``boxes'') as follows. Choose $n$ sufficiently large so that $(-n,n)^2\cap  \tilde{E}(\mathcal F)\neq\varnothing$. For the sake of simplicity, let us assume $n=1$. Let $\mathscr B_0=\{[-1,1]^2\}$.

\begin{ucl}There is a sequence of finite collections of boxes $\mathscr B_1, \mathscr B_2,\ldots$  such that for every $k\geq 1$ and  $B\in\mathscr B_k$: 
\begin{enumerate}
\item $B=[a,b]\times[c,d]$ for some  $a,b\in \mathbb R$ and $c,d\in \mathbb Q$ with $a<b$ and $c<d$,
 \item $b-a=F^{-k^2}(1)$, 
\item $d-c\leq F^{-k^2}(1)$,
\item there exists $B'\in \mathscr B_{k-1}$ such that $\{a\}\times [c,d]\subset\{b'\}\times [c',d'],$
\item $\{a\}\times [c,d]\cap J(\mathcal F)\neq\varnothing$,  
\item $B\cap B^*=\varnothing$ for all $B^*\in \mathscr B_{k}\setminus \{B\}$,
\item if $\underline s,\underline {\hat s}\in [c,d]\cap \mathbb Z ^\omega$ then $\underline s\restriction 2k^2=\underline{\hat s}\restriction 2k^2$ (i.e.\ $s_i=\hat s_i$ for all $i<2k^2$), and
\item  for every $B'\in \mathscr B_{k-1}$ and $x\in J(\mathcal F)\cap \{b'\}\times [c',d']$  there exists $B\in \mathscr B_k$ such that $x\in B$.
\end{enumerate} 
\end{ucl}

\begin{proof}Suppose $\mathscr B_{k-1}$ has already been defined.  Let $$K=\bigcup_{B'\in \mathscr B_{k-1}} \hspace{-2mm}J(\mathcal F)\cap \{b'\}\times [c',d'].$$ Since $J(\mathcal F)$ is closed in $\mathbb R^2$ \cite[Theorem 3.3]{rem} and $\mathscr B_{k-1}$ is finite, $K$ is compact.  Given a segment of integers $\langle s_0,s_1,...,s_{2k^2-1}\rangle$ of length $2k^2$, observe that the set $\{\underline s\in \mathbb Z^\omega:\underline s\restriction 2k^2=\langle s_0,s_1,...,s_{2k^2-1}\rangle\}$  is clopen and convex in the lexicographic ordering on $\mathbb Z^\omega$, and thus corresponds to  the intersection of $\mathbb P$ with an interval whose endpoints are in $\mathbb Q$.  By compactness of $K$,  the projection $$\pi_1[K]=\{y\in \mathbb R :[0,\infty)\times \{y\}\cap K\neq\varnothing\}$$  can be covered by finitely many of  those  intervals. Let $\mathcal I$ be a finite cover of $\pi_1[K]$ consisting of such intervals. Let $$-1=q_0<q_1<q_2<...<q_n=1$$ be the increasing enumeration of all   $c$'s and $d$'s used in $\mathscr B_{k-1}$, together with the endpoints  of all  intervals in $\mathcal I$. If necessary, we can insert a few more rationals to obtain   $q_{i+1}-q_i\leq F^{-k^2}(1)$ for each $i< n$. Since $\pi_1[K]$ is a compact set missing $\mathbb Q$, we can also guarantee that for every $i<n-1$, $[q_i,q_{i+1}]\cap \pi_1[K]=\varnothing$ or $[q_{i+1},q_{i+2}]\cap \pi_1[K]=\varnothing$. 

Note that $b'$ is the same for each box $B'\in \mathscr B_{k-1}$ (in fact, $b'=1+\sum_{i=1}^{k-1}F^{-i^2}(1)$). If $\{b'\}\times [q_i,q_{i+1}]\cap J(\mathcal F)\neq\varnothing$ then add the box $[b',b'+F^{-k^2}(1)]\times [q_i,q_{i+1}]$ to the collection $\mathscr B_k$.  It can be easily shown that defining $\mathscr B_k$ in this manner will satisfy conditions (1)--(8).\end{proof}



Let $\mathscr B=\bigcup \{\mathscr B_k:k<\omega\}$. For each $B=[a,b]\times [c,d]\in \mathscr B$ let $$h_B:\{a\}\times [c,d]\to ([a,b]\times \{c,d\})\cup (\{b\}\times [c,d])$$ be a homeomorphism with fixed points $\langle a,c\rangle$ and $\langle a,d\rangle$. 

\begin{ucl}\label{t9}There is a continuous mapping $$g:[0,1]\to[1,\infty)\times [-1,1]$$  such that $g(0)=\langle 1,-1\rangle$,  $g(1)=\langle 1,1\rangle$, and $g[0,1]\cap J(\mathcal F)\subset I(\mathcal F)$.\end{ucl}

 \begin{proof}

 The mapping $g$ will be the pointwise limit of a uniformly Cauchy sequence of continuous functions.  To begin, define $g_0(t)=\langle 1,2t-1\rangle$ for all $t\in [0,1]$.  Now suppose $k\geq 1$ is given and $g_{k-1}$ has been defined.   Let $t\in [0,1]$. We set $$g_k(t)=g_{k-1}(t)$$ if $g_{k-1}(t)$ is not in any member of $\mathscr B_k$.  Otherwise, by item (6) there is a unique box $B=[a,b]\times[c,d]\in \mathscr B_k$ which contains $g_{k-1}(t)$.  Inductively  $g_{k-1}(t)$ is contained in some element of $\mathscr B_0\cup\ldots\cup \mathscr B_{k-1}$, so by  (4) we have $g_{k-1}(t)\in \{a\}\times [c,d]$.   Put $$g_k(t)=h_{B}(g_{k-1}(t)).$$ The mapping $g_k$ defined in this manner is easily seen to be continuous.  \begin{figure}
\includegraphics[scale=0.6]{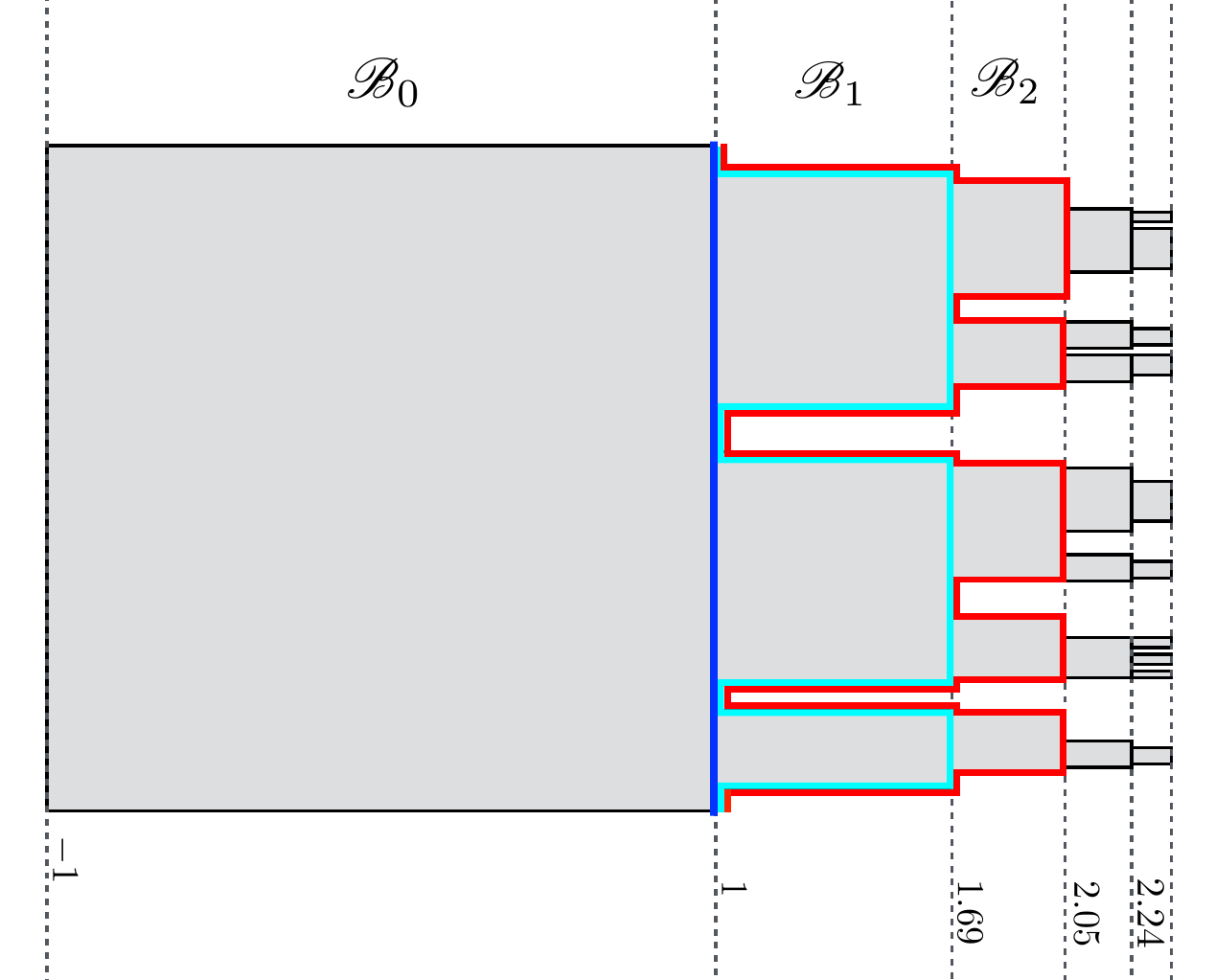}
\caption{Illustration of  $\color{blue}g_0,\color{cyan}g_1\color{black}, and\; \color{red}g_2$\color{black}.}
 \end{figure} Note that the image of $g_k$ is contained in $[1,\infty)\times [-1,1]$, $g_k(0)=\langle 1,-1\rangle$, and $g_k(1)=\langle 1,1\rangle$. 
 To see that the sequence $(g_k)$ is uniformly Cauchy, fix $\varepsilon>0$.  Let $N$ be  such that $\sum_{k=N}^\infty 1/k^2<\varepsilon/5$. 
  For any $t\in [0,1]$, if $g_k(t)\neq g_{k-1}(t)$ then $g_k(t)$ and $g_{k-1}(t)$ belong to the same element of $\mathscr B_k$. Hence, by items (2) and (3), the distance between $g_k(t)$ and $g_{k-1}(t)$ is at most $\sqrt{2}F^{-k^2}(1)$. Combined with Lemma \ref{t7} we have   $$ |g_k(t)-g_{k-1}(t)|\leq \sqrt{2}F^{-k^2}(1)<\sqrt{2}\frac{3}{k^2}<\frac{5}{k^2}$$ for every $k<\omega$ and $t\in [0,1]$. Thus for any $i> j\geq N$ and $t\in [0,1]$, $$ |g_i(t)-g_{j}(t)|\leq \sum_{k=j}^{i-1} |g_{k+1}(t)-g_{k}(t)| <\sum_{k=N}^\infty \frac{5}{k^2}<\varepsilon.$$   
This proves that $(g_k)$ is uniformly Cauchy.  Therefore $(g_k)$ converges uniformly to a continuous function  $g:[0,1]\to[1,\infty)\times [-1,1]$ which satisfies $g(0)=\langle 1,-1\rangle$ and  $g(1)=\langle 1,1\rangle$.

It remains to show that $g[0,1]\cap J(\mathcal F)\subset I(\mathcal F)$. Let $x=g(t)\in g[0,1]\cap J(\mathcal F)$. By Lemma \ref{t5} and the fact $F^k(1)\to\infty$,    to prove $x\in I(\mathcal F)$ it suffices to show  $T(\mathcal F^{2j^2}(x))\geq F^{j^2}(1)$ for all $j\geq 1$.  Fix $j\geq 1$.   To prove $T(\mathcal F^{2j^2}(x))\geq F^{j^2}(1)$, by  continuity of $T\circ \mathcal F^{2j^2}$ we only need to show that every neighborhood of $x$ contains a point $w$ such that $T(\mathcal F^{2j^2}(w))\geq F^{j^2}(1)$.

 Let $W$ be a neighborhood of   $x$. Observe that since $c$ and $d$ are rational in item (1),  the top and bottom edges of all boxes miss $J(\mathcal F)$. So  if $g_k(t)\in J(\mathcal F)$, then $g_k(t)$ belongs to the right edge of an element of $\mathscr B_{k}$.  Then by the construction of $g_{k+1}$ and items (4) and (8) we get $g_{k+1}(t)\neq g_k(t)$. Since $g_k(t)\to x$, it follows that the sequence $g_0(t),g_1(t),\ldots$  is not eventually constant. So for every $k$ there exists $B\in \mathscr B_k$ such that $g_k(t)\in B$. By (2) and (3), the diameter of each box  in $\mathscr B_k$ is at most $\sqrt{2} F^{-k^2}(1)$, which goes to $0$ as $k\to\infty$ (e.g.\ by Lemma \ref{t7}).   Since $g_k(t)\to x$ and $x$ belongs to the interior of $W$, there exists $k>j$ such that $W$  contains   a box $B(k)\in \mathscr B_k$. By (5) there exists $w\in B(k)\cap W\cap J(\mathcal F)$. 

 By (4),  for each $i<k$ there exists a unique  $B(i)\in \mathscr B_i$ such that $\pi_1[B(k)]\subset \pi_1[B(i)]$. By (4) and (5) there exists $y\in B(j)\cap B(j-1)\cap J(\mathcal F)$. Then $y$ is on the left edge of $B(j)$. Let $z$ be the point on the right edge of $B(j)$ such that $\underline s(z)=\underline s(y)$ (see Figure 3). Then $z\in J(\mathcal F)$, and by (2) we have $T(z)-T(y)=  F^{-j^2}(1)$. Thus,     by Lemma \ref{t6} $$T(\mathcal F^{2j^2}(z))\geq  F^{2j^2}(T(z)-T(y))=F^{2j^2}(F^{-j^2}(1))=F^{j^2}(1).$$ Note also that $\underline s(z)\restriction 2j^2= \underline s(w)\restriction 2j^2$ because $z$ and $w$ are contained in the same horizontal strip where the first $2j^2$ coordinates agree; see items (7)  and (4).   The preceding equations together with $T(z)\leq T(w)$ imply that $$T(\mathcal F^{2j^2}(w))\geq F^{j^2}(1),$$ as desired.\end{proof}



\begin{figure}[h]
 \includegraphics[scale=0.25]{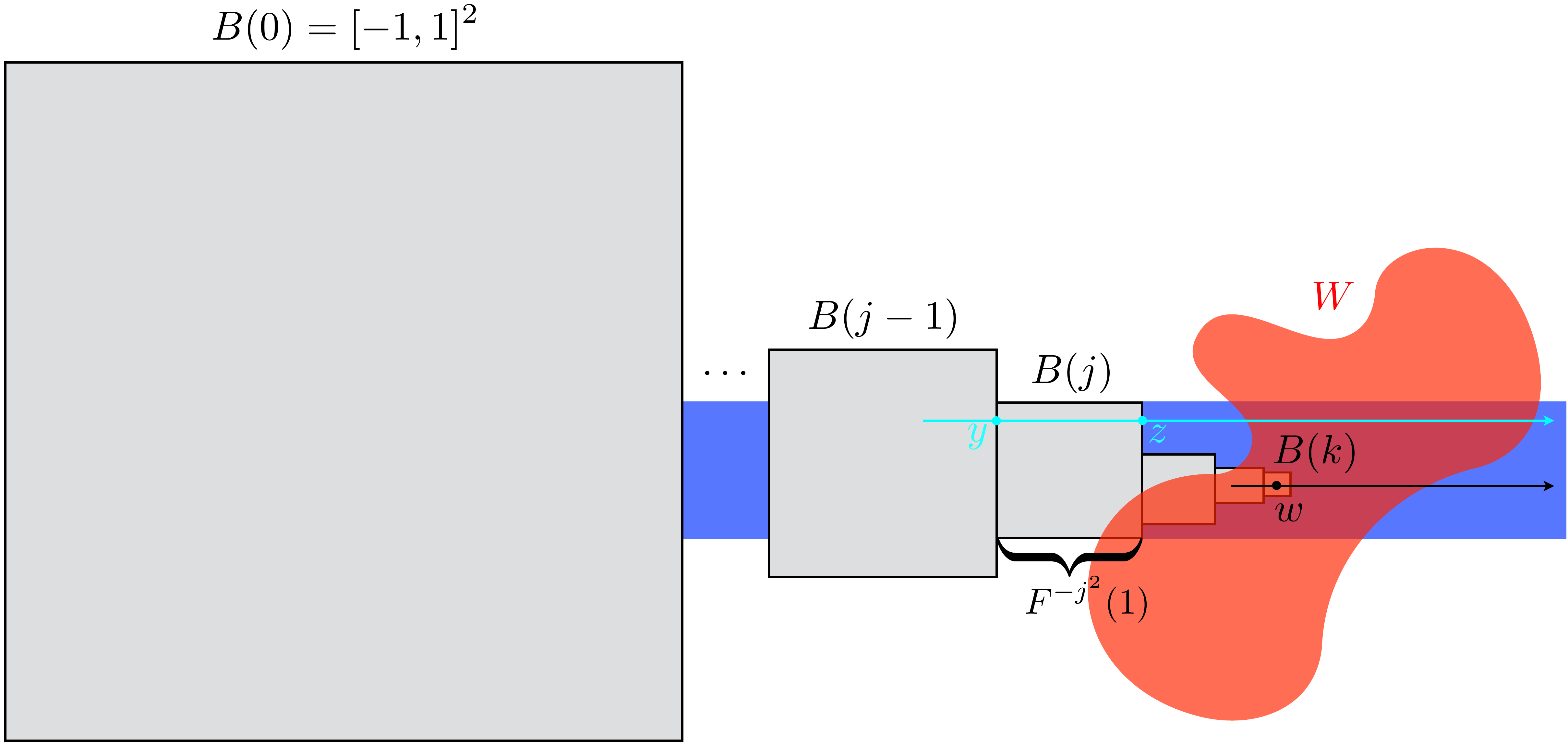}
\caption{Proof of Claim \ref{t9}.}
 \end{figure}

By \cite[Corollary 8.17]{nad} and  \cite[Theorem 8.23]{nad},  there is an arc   $\alpha:[0,1]\hookrightarrow g[0,1]$ such that $\alpha(0)=g(0)=\langle 1,-1\rangle$ and $\alpha(1)=g(1)=\langle 1,1\rangle$.  Let $\beta=(\{-1\}\times [-1,1])\cup ([-1,1]\times\{-1,1\}) \cup \alpha[0,1].$ Then $\beta$ is a simple closed curve, and by Claim \ref{t9} we have  $\beta\cap E(\mathcal F)\subset \tilde E(\mathcal F).$   If $U$ is the bounded component of $\mathbb R ^2\setminus\beta$ then $(-1,1)^2\subset U$, hence   $U\cap  \tilde{E}(\mathcal F)\neq\varnothing$.  Now we can apply Lemma \ref{lem} to see that  $\tilde{E}(\mathcal F)$ cannot be written as a countable union of nowhere dense C-sets. This concludes the proof of Theorem \ref{t1}.\hfill $\blacksquare$


\section{Remarks on Theorem 1.1}
\begin{ur}In the proof of Theorem 1.1 it is possible to show that $g$ is one-to-one, so $\alpha=g$. And the curve $\beta$ is  just the boundary of $\bigcup \mathscr B$. \end{ur}


\begin{ur}\label{rem3}The statement of Theorem 1.1 can be strengthened to: 
\begin{quote}No neighborhood in $\tilde E(\mathcal F)$ (respectively, in $\dot E(f)$) can be covered by countably many nowhere dense C-sets of $\tilde E(\mathcal F)$ (of $\dot E(f)$).\end{quote}  To see this, let $x_0\in\mathbb R ^2$ and $\varepsilon>0$. Since $J(\mathcal F)$ is closed, there is a box $[a,b]\times [c,d]\subset B(x_0,\varepsilon/2)$ containing $x_0$ such that $$(\{a\}\times [c,d]\cup[a,b]\times \{c,d\})\cap J(\mathcal F)=\varnothing.$$  Begin the construction of the $\mathscr B_k$'s with $\mathscr B_0=\{[a,b]\times [c,d]\}$,  and choose $l$ large enough so that $$\sum_{k=1}^\infty F^{-(l+k)^2}(1)<\frac{\varepsilon}{2}.$$  Construct $\mathscr B_k$ with $l+k$ replacing each   $k$ in items (2), (3) and (7).  By the arguments in Claim \ref{t9}, if $x\in J(\mathcal F)$  lies in the limit of $(g_k)$ (constructed using the new $\mathscr B_k$'s), then  $T(\mathcal F^{2n^2}(x))\geq F^{n^2}(1)$ for all $n>l$ and consequently $x\in I(\mathcal F)$. As in Claim \ref{t9} we can construct an arc $\alpha$ such that $\beta:=(\{a\}\times [c,d]\cup[a,b]\times \{c,d\})\cup \alpha$ is a simple closed curve in $B(x_0,\varepsilon)$, the bounded component of $\mathbb R ^2\setminus  \beta$ contains $x_0$, and  $\beta\cap J(\mathcal F)\subset I(\mathcal F)$. Lemma \ref{lem} now shows that $B(x_0,\varepsilon)$ cannot be covered by countably many nowhere dense C-sets of $\tilde E(\mathcal F)$.\end{ur}

\begin{ur}As argued in Section 1, Theorem 1.1 implies that the escaping endpoint set $\dot E(f_a)$ is not homeomorphic to Erd\H{o}s space $\mathfrak E$ (Corollary 1.2).\end{ur}

\section{The space $F(f_a)\cup I(f_a)$; $a\in (-\infty,-1]$}




We are now ready to prove Theorem 1.3 and Corollary 1.4. Fix $a\in (-\infty,-1]$. 

\subsection*{Proof of Theorem \ref{t3}}Let $z\in \mathbb C$. By \cite[Section 9]{rem2}  and \cite[Theorem 2.8]{rem}, there is a homeomorphism $\varphi:\mathbb R ^2\to \mathbb C$ such that $\varphi[J(\mathcal F)]=J(f_a)$ and $\varphi[I(\mathcal F)]=I(f_a)$.  Let $x=\varphi^{-1}(z)$. We have shown that there is a Jordan curve $\beta\subset(\mathbb R ^2\setminus J(\mathcal F))\cup I(\mathcal F)$ around $x$; see Remark \ref{rem3}, or simply begin the construction in Section 6 with $\mathscr B_0=\{[-n,n]^2\}$ where $n\in \mathbb N$ is such that $x\in (-n,n)^2$. Then $\varphi[\beta]\subset F(f_a)\cup I(f_a)$ is a Jordan curve which separates $z$ from $\infty$.   \hfill$\blacksquare$
 
\medskip

In the proof of Corollary \ref{t4} below, we will make use of two well-known facts: (a) $F(f_a)$ is path-connected and (b) each component of $J(f_a)$ is contained in $I(f_a)$, with the possible exception of its endpoint.

\subsection*{Proof of Corollary \ref{t4}} Let $z_0,z_1\in F(f_a)\cup I(f_a)$. We will  find a path in $F(f_a)\cup I(f_a)$ from $z_0$ to $z_1$.  There are three cases to consider. 

The case $z_0,z_1\in F(f_a)$ is trivial since $F(f_a)$ is path-connected.  

The next case is that $z_0\in F(f_a)$ and $z_1\in I(f_a)$.  Let $\beta$ be a simple closed curve around the point $z_1$, such that $\beta\cap J(f_a)\subset I(f_a)$ (apply Theorem \ref{t3}).  Let $\gamma$ be the component of $J(f_a)$ containing $z_1$. There exists $z_2\in \beta\setminus J(f_a)$ and $z_3\in \gamma\cap \beta$.  There are paths $\alpha_1\subset F(f_a)$ from $z_0$ to $z_2$,  $\alpha_2\subset \beta$ from $z_2$ to $z_3$, and  $\alpha_3\subset \gamma$ from $z_3$ to $z_1$.  It is clear that $\alpha_3$ can be constructed to avoid the endpoint of $\gamma$, so that $\alpha_3\subset I(f_a)$. Then $\alpha_1\cup \alpha_2\cup \alpha_3\subset F(f_a)\cup I(f_a)$ contains a path from $z_0$ to $z_1$. 

The third and final case $z_0,z_1\in I(f_a)$ can be handled by connecting  each point $z_0$ and $z_1$ to a third point of  $F(f_a)$, as was done the second case.\hfill$\blacksquare$

\end{document}